\documentclass[10pt]{amsart} 
\usepackage{amsmath,amssymb, amsthm,amsbsy} 
\usepackage[dvips]{graphicx}
\usepackage[usenames,dvipsnames]{pstricks} 
\usepackage{epsfig}
\usepackage{pst-grad} 
\usepackage{pst-plot} 
\newcommand{\R}{{\mathbb R}}
\newcommand{\N}{{\mathbb N}} 

\newcommand{\C}{{\mathbb C}}

\newcommand{\eps }{\varepsilon}

 \renewcommand{\geq }{\geqslant}
 \renewcommand{\leq }{\leqslant}

\usepackage{mathtools}
\DeclarePairedDelimiter{\abs}{\lvert}{\rvert}

\DeclarePairedDelimiter{\norma}{\lVert}{\rVert} 
\usepackage{braket}
\usepackage{color} 
\newcommand{\Rn}{{\mathbb R^n}}
\newenvironment{sistema}%
{\left\lbrace\begin{array}{@{}l@{}}}%
{\end{array}\right.}  

\numberwithin{equation}{section}
\newtheorem{theorem}{Theorem}[section]

\newtheorem{lemma}[theorem]{Lemma}

\theoremstyle{definition}

\newtheorem{remark}[theorem]{Remark}

\begin{document} 

\title[Sharp Hardy uncertainty principle in magnetic fields]{Sharp Hardy uncertainty principle and gaussian profiles of covariant Schr\"odinger evolutions}


\author{B.~Cassano}
\address{Biagio Cassano: SAPIENZA Universit$\grave{\text{a}}$ di Roma, Dipartimento di Matematica, P.le A. Moro 5, 00185-Roma, Italy}
\email{cassano@mat.uniroma1.it}

\author{L.~Fanelli}
\address{Luca Fanelli: SAPIENZA Universit$\grave{\text{a}}$ di Roma, Dipartimento di Matematica, P.le A. Moro 5, 00185-Roma, Italy}
\email{fanelli@mat.uniroma1.it}


\thanks{
The two authors were supported by the Italian project FIRB 2012 {\it Dispersive Dynamics: Fourier Analysis and Calculus of Variations}.}

\begin{abstract}
We prove a sharp version of the Hardy uncertainty principle for Schr\"odinger equations with external bounded electromagnetic potentials, based on logarithmic convexity properties of Schr\"odinger evolutions. We provide, in addition, an example of a real electromagnetic potential which produces the existence of solutions with critical gaussian decay, at two distinct times.
\end{abstract}

\date{\today}
\maketitle

\section{Introduction}\label{sec.Introduction}
This paper is concerned with the sharpest possible gaussian decay, at two distinct times, of solutions to Schr\"odinger equations of
the type
\begin{equation}\label{eq:main}
  \partial_t u = i (\Delta_A + V)u,
\end{equation}
where $u=u(x,t):\R^n \times  [0,1]\to\C$, and
\begin{align*} 
   &V=V(x,t)\colon \Rn\times [0,1] \to \C, \\ 
   &\Delta_A:= \nabla_A^2,\quad
  \nabla_A:=\nabla -i A,\quad A=A(x)\colon \Rn\to\Rn.
\end{align*}
We follow a program which has been developed in the magnetic free case $A\equiv 0$, in the recent years,
by Escauriaza, Kenig, Ponce, and Vega
in the sequel of papers \cite{EKPV0,EKPV1,EKPV2,EKPV3,EKPV4},
and  with Cowling in \cite{CEKPV}.
One of the main motivations is the connection with the \textit{Hardy uncertainty principle}, which can be stated as follows:

{\it if $f(x)=O\left(e^{-|x|^2/\beta^2}\right)$ and its Fourier transform
$\hat f(\xi)=O\left(e^{-4|\xi|^2/\alpha^2}\right)$, then}
\begin{align*}
  \alpha\beta<4
  &
  \Rightarrow
  f\equiv0
  \\
  \alpha\beta=4
  &
  \Rightarrow
  f\ is\ a \ constant\ multiple\ of\ e^{-\frac{|x|^2}{\beta^2}}.
\end{align*}
The solving formula for solutions to the free Schr\"odinger equation with initial datum $f$ in $L^2$, namely
\begin{equation*}
  u(x,t):=e^{it\Delta}f(x) = (2\pi it)^{-\frac{n}{2}}e^{i\frac{|x|^2}{4t}}
  \mathcal F\left(e^{i\frac{|\cdot|^2}{4t}}f\right)\left(\frac{x}{2t}\right),
\end{equation*}
gives a hint of the following PDE's-version of the Hardy uncertainty principle:

{\it if $u(x,0)=O\left(e^{-|x|^2/\beta^2}\right)$ and $u(x,T):=e^{iT\Delta}u(x,0)=
O\left(e^{-|x|^2/\alpha^2}\right)$, then}
\begin{align*}
  \alpha\beta<4T
  &
  \Rightarrow
  u\equiv0
  \\
  \alpha\beta=4T
  &
  \Rightarrow
  u(x,0)\ is\ a \ constant\ multiple\ of\ e^{-\left(\frac{1}{\beta^2}+\frac{i}{4T}\right)|x|^2}.
\end{align*}
The corresponding $L^2$-versions of the previous results were proved in \cite{SST} and affirm the following:
\begin{align*}
  e^{|x|^2/\beta^2}f \in L^2,\  e^{4|\xi|^2/\alpha^2}\hat f \in L^2,\
  \alpha\beta\leq4
  &
  \Rightarrow
  f\equiv0
  \\
  e^{|x|^2/\beta^2}u(x,0)\in L^2,\  e^{|x|^2/\alpha^2}e^{iT\Delta}u(x,0) \in L^2,\
  \alpha\beta\leq4T
  &
  \Rightarrow
  u\equiv0.
\end{align*}
We mention \cite{BD, FS, SS} as interesting surveys about this topic.
In the sequel of papers \cite{CEKPV,EKPV0,EKPV1,EKPV2,EKPV3,EKPV4}, the authors investigated the validity of the previous statements for zero-order perturbations of the Schr\"odinger equation of the form 
\begin{equation}\label{eq:main2}
  \partial_t u = i (\Delta + V(t,x))u.
\end{equation}
An interesting contribution of the above papers is that a purely real analytical proof of the uncertainty principle is provided, based on the logarithmic convexity properties of weighted $L^2$-norms of solutions to \eqref{eq:main2}.
Namely, norms of the type $H(t):=\norma{e^{a(t)\abs{x+b(t)}^2}u(t)}_{L^2(\Rn)}$, where $a(t)$ is a suitable bounded function, and $b(t)$ is a curve in $\R^n$
 are logarithmically convex in time.
The interest of these results relies on various motivations. First, since just real analytical techniques are involved, rough potentials $V\in L^\infty$ can be considered, which are usually difficult to handle by Fourier techniques.
In addition, in \cite{EKPV3} it is shown that a gaussian decay at times $0$ and $T$ of solutions to \eqref{eq:main2}  is not only preserved, but also improved, in some sense, for intermediate times, up to suitably move the center of the gaussian.
A consequence of Theorem 1 in \cite{EKPV3} is the following: if $V(t,x)\in L^\infty$ is the sum of a real-valued potential $V_1$ and a sufficiently regular complex-valued potential $V_2$, and $\norma{e^{\abs{x}^2/\beta^2}u(0)}_{L^2}
+\norma{e^{\abs{x}^2/\alpha^2}u(T)}_{L^2}<+\infty$, with $\alpha\beta<4T$, then $u\equiv0$. Moreover, the result is sharp in the class of complex potentials: indeed, Theorem 2 in \cite{EKPV3} provides an example of a (complex) potential $V$ for which there exists a non-trivial solution $u\neq0$ with the above gaussian decay properties, with $\alpha\beta=4T$.

The fact that the potential in \cite{EKPV3} is complex-valued might have an appealing connection with the examples by Cruz-Sampedro and Meshkov in \cite{C, M} about unique continuation at infinity for stationary Schr\"odinger equations. 
In particular, an interesting question is still open, concerning with the possibility or not of providing analogous real-valued examples. 

Our first result states the following: if one introduces a magnetic potential in the hamiltonian, then real-valued examples in the spirit of Theorem 2 in \cite{EKPV3} can be found.
\begin{theorem}\label{thm:example}
  Let $n=3$, $k>3/2$, and define $A=A(x,y,z,t):\R^{3+1}\to\R^3$ and $V=V(x,y,z,t):\R^{3+1}\to\R^3$ as follows:
  \begin{align}
  A(x,y,z,t) 
  &
  =
  \frac{2kt}{1+t^2}\cdot\frac{z}{(x^2+y^2)(1+r^2)}\left(xz,yz,-x^2-y^2\right),
  \label{eq:exA}
  \\
  V(x,y,z,t)
  &
  =
  \frac{1}{1+r^2}\left(\frac{2k}{1+t^2}+6k-\frac{4k(1+k)r^2}{1+r^2}
  -\left|A(x,y,z,t)\right|^2\right),
  \label{eq:exV}
  \end{align}
  where $r^2:=x^2+y^2+z^2$.
  Then the function
  \begin{equation}\label{eq:exu}
    u=u(r,,t)=(1+it)^{2k-\frac n2}(1+r^2)^{-k}e^{-\frac{(1-it)}{4(1+t^2)}r^2}
  \end{equation}
  is a solution to
  \begin{equation*}
    i\partial_tu+\Delta_Au = Vu
  \end{equation*}
  satisfying $\left\|e^{\frac{r^2}8}u(-1)\right\|_{L^2}+\left\|e^{\frac{r^2}8}u(1)\right\|_{L^2}<\infty$.
\end{theorem}
\begin{remark} \label{rem:scaling} 
  The choice of the time interval $[-1,1]$ instead of $[0,T]$ does
  not led the generality of the result, since by scaling one can always reduce matters to this case (see also Remark \ref{rem:scaling2} below).
  Notice that both $A$ and $V$ are real-valued, and this is (at our knowledge) a novelty.
  Observe moreover that $A$ is time-dependent, and singular all over the $z$-axis $x=y=0$, with Coulomb-type singularity $(x^2+y^2)^{-\frac12}$.
  We finally remark that we are not able to generalize the above example to any dimension $n\neq3$, and it is unclear to us if this is an intrinsic obstruction or not.
  The main idea relies in the expansion
  \begin{equation*}
    \Delta_A = \Delta -2iA\cdot\nabla-i\text{div}\,A-|A|^2.
  \end{equation*}
  Applying this operator to the function $u$ in \eqref{eq:exu}, one notice that the first order term $2iA\cdot\nabla u$ vanishes, since $u$ is radial and we choose the Cr\"onstrom gauge $A\cdot x\equiv0$; on the other hand a purely imaginary, non null zero-order term $i\text{div}\,A$ naturally appears, being $A$ real valued. We refer to section \ref{sec:1} below for the details of the proof, which is a quite simple computation. 
\end{remark}
Theorem \ref{thm:example} motivates us to think to electromagnetic Schr\"odinger evolutions as a natural setting for the study of Hardy uncertainty principles.
We also need to keep in mind the well known fact that the magnetic ground states (and hence the corresponding standing waves) have gaussian decay (see \cite{E} and the references therein). 

In the recent years, some results in the spirit of the Hardy principle appeared, concerning with generic first-order perturbations of Schr\"odinger operators. Among the others, Dong and Staubach in \cite{DS}
proved that an uncertainty property holds, under suitable assumptions on the lower order terms; nevertheless, a quantitative knowledge of the critical constant in the gaussian weights seems to be difficult to be found, due to the generality of the model. The paper \cite{DS} generalize a previous result by Ionescu and Kenig in \cite{IK}, in which unique continuation from the exterior of a ball is proved, in the same setting.

We stress that an electromagnetic field is not any first-order perturbation of a Schr\"odinger operator, since it has a peculiar intrinsic algebra which has to be taken into account. The feeling is that quantitative results could be obtained for such operators, under geometric assumptions on the magnetic field. As an example, we mention \cite{BFGRV}, where a non-sharp version of the Hardy uncertainty principle (with $\alpha\beta<2T$) in presence of (possibly large) magnetic fields has been recently proved, inspired to the techniques in \cite{EKPV2}. 
The last result of this paper improves the ones in \cite{BFGRV}, covering the sharp range $\alpha\beta<4T$. In order to settle the theorem, we need to introduce a few notations.

In the sequel,
we denote by $A=(A^1(x),\dots,A^n(x))\colon \Rn \to \Rn$ a real
vector field (magnetic potential). The magnetic field, denoted by $B\in M_{n\times n}(\R)$ is the antisymmetric gradient of $A$, namely 
\begin{equation*} 
  B=B(x)=DA(x)-DA^t(x),
  \qquad
  B_{jk}(x)=A_j^k(x)-A^j_k(x).
\end{equation*} 
In dimension $n=3$, $B$ is identified with the vector field $\text{curl}\,A$, by the elementary properties of antisymmetric matrices.
We can now state the last result of this paper.
\begin{theorem}\label{thm:main} 
  Let $n\geq 3$, and let $u\in\mathcal C([0,1];L^2(\R^n))$ be a
  solution to
  \begin{equation}\label{eq:main3hardy}
    \partial_tu = i\left(\Delta_A+V_1(x)+V_2(x,t)\right)u
  \end{equation} in $\R^n\times[0,1]$, with $A=(A^1(x),\dots,A^n(x))\in \mathcal C^{1,\varepsilon}_{\text{loc}}(\R^n;\R^n)$,
  $V_1=V_1(x):\R^n\to\R$, \mbox{$V_2=V_2(x,t):\R^{n+1}\to\C$}.
  Moreover, denote by $B=B(x)=DA-DA^t$, $B_{jk}=A^k_j-A^j_k$ and assume
  that there exists a unit vector $\xi\in\mathcal S^{n-1}$ such that
  \begin{equation}\label{hypo:xi} 
    \xi^tB(x)\equiv0.
    \end{equation} 
  Finally, assume that
  \begin{align}
    & \|x^tB\|_{L^\infty}^2 <\infty  \label  {eq:thm1} \\
     & \|V_1\|_{L^\infty} <\infty \label{eq:assV3hardy} \\
     & \sup_{t\in[0,1]}\left\|e^{\frac{|\cdot|^2}{(\alpha t+\beta
          (1-t))^2}}V_2(\cdot,t)\right\|_{L^\infty}
    e^{\sup_{t\in[0,1]}\left\|\Im V_2(\cdot,t)\right\| _{L^\infty}}
    <\infty 
    \label{eq:thm2}
\\
     & \left\|e^{\frac{|\cdot|^2}{\beta^2}}
      u(\cdot,0)\right\|_{L^2}+ \left\|e^{\frac{|\cdot|^2}{\alpha^2}}
      u(\cdot,1)\right\|_{L^2}<\infty, \label{eq:decay3hardy}
  \end{align}
 for some $\alpha,\beta>0$. 

The following holds:
\begin{itemize}
\item if $\alpha\beta<4$ then $u\equiv0$.
\item if $\alpha\beta\geq 4$ then 
  \begin{equation}\label{eq:weightedestimate}
    \begin{split}
      \sup_{t \in[0,1]}&\left\Vert e^{a(t)\abs{\cdot}^2}u(t)\right\Vert_{L^2(\Rn)}+ 
      \left\Vert \sqrt{t(1-t)} \nabla_A(e^{a(t)+\frac{i \dot a(t)}{8a(t)}\abs{\cdot}^2}u)\right\Vert_{L^2(\Rn\times [0,1])}  \\
      & \leq N \left[\left\Vert e^{\frac{\abs{\cdot}^2}{\beta^2}}u(0)\right\Vert_{L^2(\Rn)}
        + \left\Vert e^{\frac{\abs{\cdot}^2}{\alpha^2}}u(1)\right\Vert_{L^2(\Rn)}\right],
    \end{split}
  \end{equation}
with 
\begin{equation*}
  a(t)=\frac{\alpha\beta R}{2(\alpha t + \beta (1-t))^2+2R^2(\alpha t - \beta (1-t))^2},
\end{equation*}
where $R$ is the smallest root of the equation
\begin{equation*}
  \frac{1}{2\alpha\beta}=\frac{R}{4(1+R^2)},
\end{equation*}
and $N>0$ is a constant depending on $\alpha,\beta$ and $\norma{V}_{L^\infty(\Rn\times[0,1])}$, $\norma{x^t B}_{L^\infty(\Rn)}$. 
\end{itemize}
\end{theorem}
\begin{remark}
Notice that, apart from the local regularity assumption $A\in \mathcal C^{1,\varepsilon}_{\text{loc}}$, which is the minimal request in order to justify an approximation argument in Lemma \ref{lem:lemma4} below, all the hypotheses of Theorem \ref{thm:main} are in terms of $B$ and $V$, respectfully of the gauge invariance of the result.
It is possible to prove, by standard perturbation theory (see e.g. Proposition 2.6 in \cite{BFGRV}), that, under the assumptions of Theorem \ref{thm:main}, the operator $-\Delta_A-V_1$ is self-adjoint on $L^2$, with form domain $H^1(\R^n)$; this fact will be always implicitly used in the rest of the paper.
\end{remark}
\begin{remark}\label{rem:scaling2}
  The choice of the time interval $[0,1]$ does
  not led the generality of the results. Indeed, $v\in
  C([0,T],L^2(\Rn))$ 
  is solution to \eqref{eq:main} in \mbox{$\Rn\times[0,T]$}
  if and only if $u\colon [0,1]\to \C$, $u(x,t)=T^{\frac{n}{4}}v(\sqrt{T}x,Tt)$ is solution to
\begin{equation*}
  \partial_t u = i(\Delta_{A_T}u + V_T (x,t) u), \text{ in }[0,1]\times \Rn,
\end{equation*} 
where
\begin{equation*} 
A_T(x,t)=\sqrt{T}A(\sqrt{T}x,Tt),\quad V_T(x)=T
  V(\sqrt{T}x,Tt).
\end{equation*} 
Moreover, observe that
\begin{equation*}
  \begin{split}
    \norma{e^{\frac{\abs{x}^2}{\beta^2}}v(0)}=\norma{e^{\frac{\abs{x}^2}{\beta'^2}}u(0)}, & \quad
    \norma{e^{\frac{\abs{x}^2}{\alpha^2}}v(T)}=\norma{e^{\frac{\abs{x}^2}{\alpha'^2}}u(1)},  \\
    \sup_{t\in[0,T]} \norma{e^{\frac{T^2\abs{x}^2}{(\alpha t +
          \beta(T-t))^2}}v(t)}= & \sup_{t\in[0,1]}
    \norma{e^{\frac{\abs{x}^2}{(\alpha' t + \beta' (1-t))^2}}u(t)}.
  \end{split}
\end{equation*}
with $\beta'=T^{-\frac{1}{2}}\beta$, $\alpha'=T^{-\frac{1}{2}}\alpha$.  
\end{remark} 
\begin{remark}
  The magnetic field in Theorem \ref{thm:main} does not depend on time, differently from the example in Theorem \ref{thm:example}. Nevertheless, it could be probably possible to generalize the result to the case of time dependent magnetic fields, by assuming the existence of the purely magnetic flow and the $L^2$-preservation, but this will not be an object of study in the present paper. 
\end{remark}
\begin{remark}
  Assumption \eqref{hypo:xi} is fundamental in our strategy of proof, and it does not allow to include the 2D-case in the statement of Theorem \ref{thm:main}, due to elementary properties of antisymmetric matrices. We mention \cite{BFGRV} for some explicit examples of magnetic fields satisfying \eqref{hypo:xi}.
  It is an interesting open question if there exist examples of magnetic fields which do not satisfy \eqref{hypo:xi}, for which the Hardy uncertainty holds with different quantitative constants or different exponential decays.
  Observe that the example in \eqref{eq:exA} satisfies \eqref{hypo:xi}, with $\xi=(0,0,1)$. Indeed, an explicit computation shows that
  \begin{equation*}
    B=\text{curl}\,A=\frac{2kt}{1+t^2}\cdot\frac{2z}{(x^2+y^2)(1+r^2)^2}\left(-y, x, 0\right).
  \end{equation*}
\end{remark}
The strategy of the proof of Theorem \ref{thm:main} is the following:
\begin{enumerate}
\item first we reduce to the Cr\"onstrom gauge $x\cdot A \equiv 0$ (see Section \ref{sec:cronstrom}), which turns out to be a helpful choice;
\item by conformal (or Appell) transformation (see Lemma \ref{lem:appell}), we reduce to the case 
  $\alpha=\beta$, and we perform a time scaling to reduce to the time interval $[-1,1]$ (see Section \ref{sec:Appell}); \label{step2}
\item we prove Theorem \ref{thm:main} in the case $\mu:=\alpha=\beta$ (see Section \ref{sec:conclusion}); \label{step3}
\item we translate the result in terms of the original solution, by inverting the transformations at step \ref{step2},
  obtaining the final result.
\end{enumerate}
The key ingredient is Lemma \ref{lem:lemma4} below, which comes into play in the proof of step \ref{step3}. This is based on an iteration scheme, introduced in \cite{EKPV3}: by 
successive approximations, we can start an iterative improvement of the decay assumption \eqref{eq:decay3hardy}, by suitably moving the center of the gaussian weight. In the limit, this argument leads to an optimal choice of the function $a=a(t)\colon [-1,1]\to
\R$ for which the estimate
\begin{equation}
  \label{eq:stimapesata}
  \norma{e^{a(t)\abs{x}^2}u(x,t)}_{L^\infty_t([-1,1])L^2_x(\Rn)}\leq
  C(\alpha,\beta,T,\norma{V}_{L^\infty},\norma{x^t B}_{L^\infty})
\end{equation} 
holds. The presence of a magnetic fields makes things quite more complicate, once the iteration starts, as wee see in the sequel.
The rest of the paper is devoted to the proofs of our main theorems. 

{\bf Acknowledgements.} We wish to thank Simone Cacace for some useful numerical simulations about the iteration scheme in Section \ref{sec:conclusion} below. We also wish to thank Gustavo Ponce for addressing reference \cite{E} to us.  

\section{Proof of Theorem \ref{thm:example}}\label{sec:1}
The proof of Theorem \ref{thm:example} is a straightforward computation.
First, we expand the magnetic Laplace operator and rewrite
\begin{equation*}
  \left(i\partial_t+\Delta_A\right)u=
  \left(i\partial_t+\Delta\right)u 
  -2iA\cdot\nabla_x u-i(\text{div}_xA)u-|A|^2u.
\end{equation*}
Now we compute
\begin{align*}
  \left(i\partial_t+\Delta\right)u 
  &
  = 
  \frac1{1+r^2}\left(\frac{2k}{1+it}+6k-\frac{4k(k+1)r^2}{1+r^2}\right)
  \\
  &
  =
  \frac1{1+r^2}\left(-\frac{2ikt}{1+t^2}+\frac{2k}{1+t^2}+6k-\frac{4k(k+1)r^2}{1+r^2}\right),
\end{align*}
where $u$ is given by \eqref{eq:exu}.
Observe that, since $u$ is radial and $A\cdot x\equiv0$ by the definition \eqref{eq:exA}, we have $A\cdot\nabla_xu\equiv0$.
Finally, another direct computation gives
\begin{equation*}
  i\text{div}_xA = -\frac{2ikt}{1+t^2}\cdot\frac{1}{1+r^2}.
\end{equation*}
In conclusion, 
\begin{equation*}
  \left(i\partial_t+\Delta_A\right)u=
  \frac1{1+r^2}\left(\frac{2k}{1+t^2}+6k-\frac{4k(k+1)r^2}{1+r^2}-|A|^2\right)
  =Vu,
\end{equation*}
by the definition \eqref{eq:exV}, which completes the proof.

The rest of the paper is devoted to the proof of Theorem \ref{thm:main}.

\section{Some preliminary lemmata}\label{sec:prelim}

Let us fix some notations and recall some results from
\cite{EKPV3} and \cite{BFGRV}. We denote
\begin{equation*} (f,g):=\int_{\Rn} f \bar g \, dx, \quad
  H(f)=\norma{f}^2:=(f,f),
\end{equation*}
for $f,g\in L^2(\R^n)$.

\begin{lemma}[\cite{EKPV3}, Lemma 2]\label{lemmaastratto} Let
  $\mathcal{S}$ be a symmetric operator, $\mathcal{A}$ a skew-symmetric
  one, both allowed to depend on the time variable, and $f$ a smooth enough 
  function. Moreover let \mbox{$\gamma\colon [c,d]\to (0,+\infty)$} and
  $\psi \colon [c,d] \to \R$ be smooth functions. If 
  \begin{equation}\label{eq:lemmaastratto} (\gamma \, \mathcal{S}_t
    f(t)+\gamma \, [\mathcal{S},\mathcal{A}] f(t) + \dot{\gamma} \,
    \mathcal{S} f(t),f(t)) \geq -\psi(t) H(t), \quad t\in[c,d]
  \end{equation} 
then, for all $\varepsilon >0$,
  \begin{equation*} 
    H(t)+\eps \leq
    e^{2T(t)+M_\eps(t)+2N_\eps(t)}(H(c)+\eps)^{\theta(t)}(H(d)+\eps)^{1-\theta(t)},\quad
    t\in[c,d]
    \end{equation*}
where $T$ and $M_\eps$ verify
  \begin{equation*}
    \begin{sistema}
      \partial_t(\gamma \partial_t T)=-\psi,\quad t\in[c,d] \\ T(c)=T(d)=0,
    \end{sistema} \quad
    \begin{sistema}
      \partial_t (\gamma \partial_t M_\eps)=-\gamma \frac{\norma{\partial_t
          f -\mathcal{S}f- \mathcal{A}f}^2}{H+\eps} \quad t\in[c,d]\\
      M_\eps(c)=M_\eps(d)=0,
    \end{sistema}
  \end{equation*}
  \begin{equation*} N_\eps=\int_c^d \left\vert \Re
      \frac{((\partial_s-\mathcal{S}-\mathcal{A})f(s),f(s))}{H(s)+\eps}\right\vert\,ds,
    \quad \theta(t)=\frac{\int_t^d \frac{ds}{\gamma} }{\int_c^d
      \frac{ds}{\gamma}}.
  \end{equation*} 
Moreover
  \begin{equation}\label{eq:astrattoutile}
    \begin{split}
      \partial_t&(\gamma \, \partial_t H -\gamma \, \Re(\partial_t f -
      \mathcal{S} f- \mathcal{A}f,f))+ \gamma \, \norma{\partial_t f
        -\mathcal{S}f-\mathcal{A}f}^2 \\ &\geq 2 (\gamma \, \mathcal{S}_t
      f + \gamma \, [\mathcal{S},\mathcal{A}]f + \dot \gamma \,
      \mathcal{S}f,f).
    \end{split}
  \end{equation}
\end{lemma}

For
$\varphi=\varphi(x,t):\R^{n+1}\to\R$, we can write
\begin{equation*} 
  e^{\varphi(x,t)}(\partial_t -
  i\Delta_A)e^{-\varphi(x,t)}= \partial_t - \mathcal{S}- \mathcal{A}
\end{equation*} 
where
\begin{align}
  \label{eq:S} \mathcal S & =
  -i\left(\Delta_x\varphi+2\nabla_x\varphi\cdot\nabla_A\right)
  +\varphi_t \\
  \label{eq:A} \mathcal A & = i
  \left(\Delta_A+|\nabla_x\varphi|^2\right)
\end{align} 
(see \cite{BFGRV}). Observe that $\mathcal{S}$ and $\mathcal A$ are
respectively a symmetric and a skew-symmetric operator.  
Our first goal is to apply Lemma \ref{lemmaastratto} with a suitable choice of $\mathcal S,\mathcal A$. In order to do this, we need to obtain the lower bound \eqref{eq:lemmaastratto} when $\mathcal S$ and $\mathcal A$ are given by \eqref{eq:S} and \eqref{eq:A}, respectively: this is done in the following lemma, analogous to \cite{EKPV3}, Lemma 3.

\begin{lemma}\label{lemma3}
Let 
\begin{equation}\label{eq:varphi}
  \begin{split}
  &\varphi(x,t)=a(t)\abs{x+\bold b(t)}^2,\\
  &a=a(t) \colon \R \to \R,  \quad \mathbf{b}=\mathbf{b}(t)=b(t)\xi\colon \R \to \Rn,
  \quad
  \xi\in\mathbb S^{n-1},
  \end{split}
\end{equation}
and $\mathcal S,\mathcal A$ be defined as in \eqref{eq:S} and \eqref{eq:A}. Assume that
 \begin{equation}
   \label{eq:ipotesilemmaastratto4}
   \begin{split}
     x \cdot A_t(x)&= 0, \\
   \mathbf{b} \cdot A_t(x)&=0,
 \end{split}
\end{equation}
for all $x \in \R^n$ and assume \eqref{hypo:xi}.
Assume moreover 
  \begin{equation} 
     F(a,\gamma)=\gamma\left( \ddot a + 32 a^3
      -\frac{3\dot a^2}{2a} -\frac{a}{2}\left(\frac{\dot a}{a}+\frac{\dot
          \gamma}{\gamma}\right)^2\right) >0 \text{ in }[c,d].
  \end{equation} 
  Then, for a smooth enough function $f$, 
  \begin{equation} 
    ((\gamma \mathcal{S}_t + \gamma
    [\mathcal{S},\mathcal{A}]+\dot \gamma \mathcal{S})f,f) \geq -\left(\left(
      \frac{\gamma^2 a^2 \abs{\mathbf{\ddot b}}^2}{F(a,\gamma)} +2 \gamma a \norma{x^t B}_{L^\infty}^2\right)f,f\right), 
     \quad \text{ for all } t \in [c,d].
  \end{equation}
\end{lemma}
\begin{proof} The proof is analogous to the one of Lemma 3 in 
  \cite{EKPV3}, with some additional magnetic terms to be considered. Explicit
  computations (see Lemma 2.9 in \cite{BFGRV}) give:
  \begin{align*} 
      \int_{\Rn}\mathcal{S} f \bar f\,dx =
      & \int_{\Rn} \left[ -i(2n a \abs{f}^2+4a(x+\mathbf{b})\cdot\nabla_A f \bar f)\right]\,dx \\
      &+\int_{\Rn}\left[\dot a \abs{x+\mathbf{b}}^2\abs{f}^2 +2 a \mathbf{\dot b} \cdot (x+\mathbf{b})\abs{f}^2 \right] \, dx\\
       \int_{\Rn}\mathcal{A} f \bar f \,dx =
      & \int_{\Rn}\left[(i \Delta_A f+4i a^2 \abs{x+\mathbf{b}}^2 f)\bar f \right]\,dx \\
       \int_{\Rn}[\mathcal{S},\mathcal{A}]f \bar f \,dx=
      & \int_{\Rn}\left[ 8a \abs{\nabla_A f}^2 + 32a^3 \abs{x+\mathbf{b}}^2 \abs{f}^2\right]\,dx \\
      & -\int_{\Rn}\left[4\Im [f \, 2a(x+\mathbf{b})^t B \overline{\nabla_A f} ] \right]\,dx \\
       &+ \int_{\Rn} \left[ 4 \Im[\bar f \dot a (x+\mathbf{b})\cdot \nabla_A f +\bar f a \mathbf{\dot b} \cdot \nabla_A f] \right] \, dx
     \end{align*}
     \begin{align*}
       \int_{\Rn}\mathcal{S}_t f \bar f \,dx=
       & \int_{\Rn} \left[ 2\Im[(2\dot a(x+\mathbf{b})+2a\mathbf{\dot b})\cdot\nabla_A f]\bar f\right]\,dx\\ 
       & +\int_{\Rn} \left[\ddot a \abs{x+\mathbf{b}}^2- 4 a (x+\mathbf{b})\cdot A_t\right]\abs{f}^2  \,dx \\ 
      &+\int_{\Rn} \left[ 4 \dot a \mathbf{\dot b} \cdot (x+\mathbf{b}) + 2 a  \mathbf{ \ddot b} \cdot (x+\mathbf{b})+2a \abs{\mathbf{\dot b}}^2 \right]\abs{f}^2 \,
    dx.
  \end{align*} 
Summing up we get
  \begin{align*} 
    \int_{\Rn} (&\gamma \mathcal{S}_t + \gamma [\mathcal{S},\mathcal{A}]+\dot \gamma \mathcal{S})f \bar f\, dx \\
               =& \int_{\Rn}(\ddot a \gamma + 32 a^3 \gamma + \dot \gamma \dot a)\abs{x+\mathbf{b}}^2 \abs{f}^2\,dx \\ 
                &+ \int_{\Rn}[(4 \gamma \dot a \mathbf{\dot b} + 2 \gamma a \mathbf{\ddot b} 
                 + 2 \dot \gamma a \mathbf{\dot b})\cdot (x+\mathbf{b})+2\gamma a \abs{\mathbf{\dot b}}^2]\abs{f}^2\,dx \\
                &+ \int_{\Rn}8 \gamma a \abs{\nabla_A f}^2 + 2 \Re (-i \nabla_A f)\cdot \overline{(4 \gamma a \mathbf{\dot b} f)}\,dx \\ 
                &+ \int_{\Rn}2\Re (-i\nabla_A f)\cdot \overline{((2\dot \gamma a + 4 \gamma \dot a)(x+\mathbf{b})f)}\,dx \\ 
                & - \int_{\Rn}4 \Im(\gamma f 2a (x+\mathbf{b})^t B \overline{\nabla_A f})\,dx  \\
                & - \int_{\Rn}4 a \gamma \, (x+\mathbf{b})\cdot A_t \abs{f}^2\,dx .\\ 
  \end{align*} 
The last term in the previous equation vanishes,
because of \eqref{eq:ipotesilemmaastratto4}. Completing the squares in
the previous equation we get
\begin{equation}\label{eq:richiamero}
    \begin{split} 
      \int_{\Rn} (&\gamma \mathcal{S}_t + \gamma [\mathcal{S},\mathcal{A}]+\dot \gamma \mathcal{S})f \bar f\, dx  \\
                 =& \int_{\Rn} 8\gamma a \left\vert -i\nabla_A f + \frac{\mathbf{\dot b}}{2}f+ \left(\frac{\dot a }{2a}+
                     \frac{\dot \gamma}{4\gamma}\right)(x+\mathbf{b})f \right\vert ^2\,dx \\ 
                  &+\int_{\Rn}F(a,\gamma)\left\vert x+\mathbf{b}+\frac{a\gamma \mathbf{\ddot b}}{F(a,\gamma)}\right\vert^2\abs{f}^2\,dx
                   -\frac{\gamma^2 a^2 \abs{\mathbf{\ddot b}}^2}{F(a,\gamma)}\int_{\Rn}\abs{f}^2\,dx \\ 
                  &-8\gamma a \int_{\Rn}\Im(f(x+\mathbf{b})^t B \overline{\nabla_A f})\,dx.
    \end{split}
\end{equation} 
Thanks to hypotesis \eqref{hypo:xi} and the fact that
  $B$ is anti-symmetric we have
\begin{equation*} 
    f(x+\mathbf{b})^t B \overline{\nabla_A f} = f x^t B
    \overline{\nabla_A f} = f x^t B \overline{\left(\nabla_A f + 
     i \frac{\mathbf{\dot b}}{2}f+ i\left(\frac{\dot a}{2a}
    + \frac{\dot \gamma}{4\gamma}\right)(x+\mathbf{b})f\right)},
\end{equation*}
for almost all $x \in \Rn, t \in [0,1]$.


We can finally estimate
\begin{equation*}
    \begin{split} 
      8 \gamma a &\,\Im\int_{\Rn}f (x+\mathbf{b})^t B \overline{\nabla_A f}\,dx  \\
      &=8 \gamma a \Re\int_{\Rn}f x^t B \overline{\left(-i \nabla_A f
        + \frac{\mathbf{\dot b}}{2}f+ \left(\frac{\dot a}{2a}+\frac{\dot\gamma}{4\gamma}\right)(x+\mathbf{b})f\right)}\,dx \\ 
       &\leq  \, 2 \gamma a \norma{x^t B}_{L^\infty}^2 \int_{\Rn}\abs{f}^2\,dx \\ 
       &\quad +8 \gamma a \int_{\Rn}\left\vert -i\nabla_A f + \frac{\mathbf{\dot b}}{2}f
        +\left(\frac{\dot a}{2a}+\frac{\dot\gamma}{4\gamma}\right)(x+\mathbf{b})f\right\vert^2\,dx,
    \end{split}
\end{equation*} 
which proves the result.
\end{proof}

We now choose
\begin{equation*} \gamma:=a^{-1},
\end{equation*}
hence
\begin{equation*} F(a):=F(a,\gamma)= \frac{1}{a}\left( \ddot a + 32
    a^3 - \frac{3\dot a^2}{2a}\right).
\end{equation*}

The next result is the key ingredient in the proof of our main Theorem \ref{thm:main}. Its magnetic-free version $B\equiv0$ has been proved in \cite{EKPV3}.
\begin{lemma}[improved decay]\label{lem:lemma4}
Let $u \in C([-1,1],L^2(\Rn))$ be a solution to
\begin{equation*}
  \partial_t u = i(\Delta_A u + V(x,t)u) \quad \text{ in }\Rn\times[-1,1],
\end{equation*}
with $V$ a bounded complex-valued potential and $A\in W^{1+\eps,\infty}_{\text{loc}}(\Rn)$.
Assume that, for some $\mu>0$,
\begin{equation}\label{eq:check1}
  \sup_{t \in [-1,1]}\norma{e^{\mu \abs{x}^2}u(t)}< +\infty,
\end{equation}
and, for $a\colon [-1,1] \to (0,+\infty)$, smooth, even and such that
$\dot a \leq 0$, $a(1)=\mu$, $a\geq \mu$,  and $F(a)>0$ in $[-1,1]$, we have
\begin{equation}\label{eq:check2}
  \sup_{t\in [-1,1]}\norma{e^{(a(t)-\eps)\abs{x}^2}u(t)}<+\infty \quad \text{ for all }\eps >0.
\end{equation}
Then, for $\mathbf{b}=\mathbf{b}(t)=b(t)\xi \colon [-1,1]\to \Rn$, smooth, such that $\mathbf{b}(-1)=\mathbf{b}(1)=0$,
\begin{equation}
  \label{eq:spostaregaussiana}
  \norma{e^{a(t)\abs{x+\mathbf{b}(t)}^2}u(t)} \leq e^{T(t)+2\norma{V}_{L^\infty}+\frac{\norma{V}_{L^\infty}^2}{4}}
  \sup_{t\in[-1,1]}\norma{e^{\mu \abs{x}^2}u(t)}, \quad -1\leq t\leq 1,
\end{equation}
where $T$ is defined by
\begin{equation*}
  \begin{sistema}
    \partial_t\left(\frac{1}{a}\partial_t T\right)=-\left(\frac{\abs{\mathbf{\ddot b}}^2}{F(a)}
      +2\norma{x^t B}_{L^\infty}^2\right) \quad \text{ in }[-1,1]\\
    T(-1)=T(1)=0.
  \end{sistema}
\end{equation*}
Moreover, there is $C_a>0$ such that
\begin{equation}\label{eq:disuguaglianzalemma4}
  \begin{split}
    &\norma{\sqrt{1-t^2}\nabla_A(e^{a+\frac{i\dot a}{8a}\abs{x}^2}u)}_{L^2(\Rn\times[-1,1])}\\
    &+ C_a\norma{\sqrt{1-t^2}e^{a(t)\abs{x}^2}\nabla_A u}_{L^2(\Rn\times[-1,1])} \\
    &\leq e^{2\norma{V}_{L^\infty}+\frac{\norma{V}_{L^\infty}^2}{4}}\sup_{t\in[-1,1]}\norma{e^{\mu\abs{x}^2}u(t)}.
  \end{split}
\end{equation}
\end{lemma}
\begin{proof}
Extend $u$ to $\R^{n+1}$ as $u \equiv 0$ when $\abs{t}>1$ and, for $\eps>0$, set
\begin{equation*}
  a_\eps(t):=a(t)-\eps, \quad g_\eps(x,t)=e^{a_\eps(t)\abs{x}^2}u(x,t), 
\quad f_\eps(x,t)=e^{a_\eps(t)\abs{x+\mathbf{b}(t)}^2}u(x,t).
\end{equation*}
The function $f_\eps$ is in $L^{\infty}([-1,1],L^2(\Rn))$ and satisfies
\begin{equation*}
  \partial_t f_\eps -\mathcal{S_\eps}f_\eps - \mathcal{A_\eps}f_\eps=i V(x,t)f_\eps
\end{equation*}
in the sense of distribution, i.e.
\begin{equation*}
  \int_{\Rn} f_\eps \overline{(-\partial_s \zeta-\mathcal{S_\eps}\zeta+ \mathcal{A_\eps}\zeta)}\,dyds=i\int_{\Rn} V f_\eps \bar \zeta \,dyds
\end{equation*}
for all $\zeta \in C^{\infty}_0(\Rn\times (-1,1))$,
where $\mathcal{S_\eps}$ and $\mathcal{A_\eps}$ are defined as $\mathcal{S}$ and $\mathcal{A}$ are in \eqref{eq:S}, \eqref{eq:A} with $a_\eps$ in place of $a$.
We denote here $\mathcal{S}_\eps^{x,t}$, $\mathcal{A}_\eps^{x,t}$  and $\mathcal{S}_\eps^{y,s}$, $\mathcal{A}_\eps^{y,s}$ 
the operators acting on the variables $x,t$ and $y,s$ respectively.

Since all the previous results make sense for regular functions, 
the strategy is to mollify the function $f_\eps$, obtain results for the new regular function, 
and uniformly control the errors. 
Let then $\theta\in C^{\infty}(\R^{n+1})$ be a standard mollifier supported in the unit ball of $\R^{n+1}$ and,
 for $0< \delta \leq \frac{1}{4}$ set $g_{\eps,\delta}=g_\eps \ast \theta_\delta$, $f_{\eps,\delta}=f_\eps \ast \theta_\delta$, and
 \begin{equation*}
   \theta_\delta^{x,t}(y,s)=\delta^{-n-1}\theta\left(\frac{x-y}{\delta},\frac{t-s}{\delta}\right).
 \end{equation*}
The functions $f_{\eps,\delta}$ and $g_{\eps,\delta}$ are in $C^{\infty}([-1,1],\mathcal{S}(\Rn))$. 

By continuity, there exists $\eps_a >0$ such that
\begin{equation*}
  F(a_\eps)\geq \frac{F(a)}{2}, \quad \text{ in }[-1,1],
\end{equation*}
when $0<\eps \leq \eps_a$, and for such an $\eps>0$ it is possible to find $\delta_\eps>0$, with
$\delta_\eps$ approaching zero as $\eps$ tends to zero, such that
\begin{equation*}
  \left(a(t)-\frac{\eps}{2}\right)\abs{x}^2 \leq \mu \abs{x}^2, \quad
  \left(a(t)-\frac{\eps}{2}\right)\abs{x+\mathbf{b}(t)}^2 \leq \mu \abs{x}^2,
\end{equation*}
when $x\in\Rn$, $1-\delta_\eps \leq \abs{t}\leq 1$. In the following
we assume $0<\eps\leq \eps_a$ and $0<\delta\leq \delta_\eps$.

We can apply Lemma \ref{lemmaastratto} to $f_{\eps,\delta}$, with 
$H_{\eps,\delta}(t)=\norma{f_{\eps,\delta}(t)}^2,$ \mbox{$[c,d]=[-1+\delta_\eps,1-\delta_\eps]$},
$\gamma=a_\eps^{-1}$, $\mathcal{S}=\mathcal{S_\eps}$ and $\mathcal{A}=\mathcal{A_\eps}$: it turns out that
\begin{equation}\label{eq:111}
  H_{\eps,\delta}(t)\leq \left(\sup_{t\in[-1,1]}\norma{e^{\mu\abs{x}^2}u(t)}+\eps\right)^2 
  e^{2T_\eps(t)+M_{\eps,\delta}(t)+2N_{\eps,\delta}(t)}
\end{equation}
when $\abs{t}\leq 1-\delta_\eps$, and where $T_\eps$, $M_{\eps,\delta}$ and $N_{\eps,\delta}$
are defined by
  \begin{gather}
    \begin{sistema}\label{eq:222}
      \partial_t(\frac{1}{a_\eps} \partial_t T_\eps)=-\frac{\abs{\mathbf{\ddot b}}^2}{F(a_\eps)}-2\norma{x^t B}_{L^\infty}^2,
      \quad t\in[-1+\delta_\eps,1-\delta_\eps] \\ 
      T(-1+\delta_\eps)=T(1-\delta_\eps)=0
    \end{sistema}, \\
    \begin{sistema}\label{eq:333}
      \partial_t (\frac{1}{a_\eps} \partial_t M_{\eps,\delta})=-\frac{1}{a_\eps} 
      \frac{\norma{\partial_t f_{\eps,\delta} -\mathcal{S_\eps}f_{\eps,\delta}
          - \mathcal{A_\eps}f_{\eps,\delta}}^2}{H_{\eps,\delta}+\eps}, 
      \quad t\in[-1+\delta_\eps,1-\delta_\eps]\\
      M_{\eps,\delta}(-1+\delta_\eps)=M_{\eps,\delta}(1-\delta_\eps)=0
    \end{sistema}, \\
   N_{\eps,\delta}=\int_{-1+\delta_\eps}^{1-\delta_\eps} 
   \frac{\norma{(\partial_s-\mathcal{S_\eps}-\mathcal{A_\eps})f_{\eps,\delta}(s)}}
        {\sqrt{H_{\eps,\delta}(s)+\eps}}\,ds,
        \label{eq:444}
  \end{gather}

In view to let $\delta\to0$ in \eqref{eq:111}, \eqref{eq:222}, \eqref{eq:333}, \eqref{eq:444}, we compute

\begin{equation*}
  \begin{split}
    &(\partial_t f_{\eps,\delta}-\mathcal{S}_\eps^{x,t} f_{\eps,\delta} - \mathcal{A}_\eps^{x,t} f_{\eps,\delta})(x,t)\\
    =&\int_{\Rn} f_\eps(y,s)\overline{(-\partial_s\theta_\delta^{x,t}(y,s))}\,dyds
     + \int_{\Rn} (-\mathcal{S}^{x,t}_\eps - \mathcal{A}^{x,t}_\eps)f_{\eps}(y,s)\theta^{x,t}_\delta(y,s)\,dyds\\
    =&  \int_{\Rn} f_\eps(y,s)\overline{(-\partial_s-\mathcal{S}^{y,s}_\eps + \mathcal{A}^{y,s}_\eps)\theta_\delta^{x,t}(y,s)}\,dyds \\
     &+ \int_{\Rn} f_\eps(y,s)(-\mathcal{S}^{x,t}_\eps - \mathcal{A}^{x,t}_\eps+ 
       \overline{\mathcal{S}^{y,s}_\eps - \mathcal{A}^{y,s}_\eps})\theta^{x,t}_\delta(y,s)\,dyds. \\
     \end{split}
   \end{equation*}
Expliciting the term in the previous relation, we get
\begin{equation}\label{eq:definizioni}
  \begin{split}
   (\partial_t & f_{\eps,\delta}-\mathcal{S}_\eps^{x,t} f_{\eps,\delta} - \mathcal{A}_\eps^{x,t} f_{\eps,\delta})(x,t)
  \\
    =&\int_{\Rn} f_\eps(y,s) \overline{(-\partial_s -\mathcal{S}_\eps^{y,s} +\mathcal{A}_\eps^{y,s})  \theta_\delta^{x,t}(y,s)}\,dy ds\\
    &+\int_{\Rn} f_\eps(y,s) [(\dot a_\eps(s)+4ia_\eps^2(s))\abs{y+\mathbf{b}(s)}^2-(\dot a_\eps(t)+4ia_\eps^2(t))\abs{x+\mathbf{b}(t)}^2]
    \theta_\delta^{x,t}(y,s)\,dyds\\
    &+\int_{\Rn} f_\eps(y,s) [2a_\eps(s)\mathbf{\dot b}(s)\cdot (y+\mathbf{b}(s))-2a_\eps(t)\mathbf{\dot b}(t)\cdot (x+\mathbf{b}(t))] 
    \theta_\delta^{x,t}(y,s)\,dyds\\
    &+4i \int_{\Rn} f_\eps(y,s) [a_\eps(s)(y+\mathbf{b}(s))\cdot \overline{\nabla_{A,y}}
    +a_\eps(t)(x+\mathbf{b}(t))\cdot \nabla_{A,x}]  \theta_\delta^{x,t}(y,s)\,dyds\\
    &+ \int_{\Rn} 2in f_\eps(y,s) [ a_\eps(s)+a_\eps(t)] \theta_\delta^{x,t}(y,s)\,dyds \\
    &-i \int_{\Rn} f_\eps(y,s)[\Delta_{A,x}-\overline{\Delta_{A,y}}] \theta_\delta^{x,t}(y,s)\,dyds
    =:I+II+II+IV+V+VI. \\
  \end{split}
\end{equation}
Since $\nabla_x \theta^{x,t}_\delta(y,s)=-\nabla_y \theta^{x,t}_\delta(y,s)$, we have
  \begin{equation}\label{eq:IV}
  \begin{split}
        IV = &
    4i\int_{\Rn} f_\eps(y,s) [a_\eps(s)(y+\mathbf{b}(s))-a_\eps(t)(x+\mathbf{b}(t))]\cdot \nabla_{y}  \theta_\delta^{x,t}(y,s)]\,dyds
    \\
    &+ 4 \int_{\Rn} f_\eps(y,s) [- a_\eps(s)(y+\mathbf{b}(s))\cdot A(y) +  a_\eps(t) (x+\mathbf{b}(t))\cdot A(x) ]\theta^{x,t}_\delta (y,s) \,dyds.
    \end{split}
  \end{equation}
Moreover, recalling that
\begin{equation*}
\Delta_A f= \nabla_A^2 f=\Delta f-i(\nabla\cdot A)f -2i A\cdot \nabla f -\abs{A}^2 f,
\end{equation*}
and $ \Delta_{y} \theta_\delta^{x,t}(y,s) = \Delta_{x} \theta_\delta^{x,t}(y,s)$,
we obtain 
\begin{equation}\label{eq:VI}
  \begin{split}
    VI= & 
       \int_{\Rn} f_\eps(y,s)\Bigl[- (\nabla_x\cdot A(x)+\nabla_y\cdot A(y))
      +2 (A(x)- A(y))\cdot \nabla_y \Bigr.
      \\ &
      \Bigl.\ \ \ +i(\abs{A(x)}^2-\abs{A(y)}^2)\Bigr]\theta^{x,t}_\delta(y,s)\,dyds
  \end{split}
\end{equation}
By \eqref{eq:definizioni}, \eqref{eq:IV}, \eqref{eq:VI} we can hence write
\begin{equation}\label{eq:errori}
    (\partial_t f_{\eps,\delta}-\mathcal{S}_\eps^{x,t} f_{\eps,\delta} - \mathcal{A}_\eps^{x,t} f_{\eps,\delta})(x,t) = i (V f_\eps)\ast \theta_\delta(x,t)+A_{\eps,\delta}(x,t)+B_{\eps,\delta}(x,t),
\end{equation}
where
\begin{equation*}
  \begin{split}
    A_{\eps,\delta}&(x,t) \\
    =&\int_{\Rn} f_\eps(y,s) [(\dot a_\eps(s)+4ia_\eps^2(s))\abs{y+\mathbf{b}(s)}^2-(\dot a_\eps(t)+4ia_\eps^2(t))\abs{x+\mathbf{b}(t)}^2]
    \theta_\delta^{x,t}(y,s)\,dyds\\
    &+\int_{\Rn} f_\eps(y,s) [2a_\eps(s)\mathbf{\dot b}(s)\cdot (y+\mathbf{b}(s))-2a_\eps(t)\mathbf{\dot b}(t)\cdot (x+\mathbf{b}(t))] 
    \theta_\delta^{x,t}(y,s)\,dyds \\
    &+ 4 \int_{\Rn} f_\eps(y,s) [ a_\eps(t) (x+\mathbf{b}(t))\cdot A(x) - a_\eps(s)(y+\mathbf{b}(s))\cdot A(y) ]
    \theta^{x,t}_\delta (y,s) \,dyds  \\
    & +i \int_{\Rn} f_\eps(y,s) [\abs{A(x)}^2-\abs{A(y)}^2]
    \theta^{x,t}_\delta (y,s) \,dyds,    
  \end{split}
\end{equation*}
and
\begin{equation*}
  \begin{split}
    &B_{\eps,\delta}(x,t) \\
    &\,= 4i\int_{\Rn} f_\eps(y,s) [4i(a_\eps(s)(y+\mathbf{b}(s))-a_\eps(t)(x+\mathbf{b}(t)))+ 2 (A(x)-A(y))]\cdot \nabla_{y} 
    \theta_\delta^{x,t}(y,s)]\,dyds\\
    &\quad + \int_{\Rn}  f_\eps(y,s) [2in( a_\eps(s)+a_\eps(t)) - (\nabla_x\cdot A(x)+\nabla_y \cdot A(y))] \theta_\delta^{x,t}(y,s)\,dyds .
  \end{split}
\end{equation*}
Since $a_\eps$, $\mathbf b$ are smooth, and $A\in \mathcal C^{1,\varepsilon}_{\text{loc}}$, there is a $N_{a,\mathbf{b},A,\eps}>0$ such that
\begin{align}
  \norma{A_{\eps,\delta}}_{L^2(\Rn\times[-1+\delta,1-\delta])}\leq & \,
  \delta N_{a,\mathbf{b},A,\eps}\sup_{t\in[-1,1]}\norma{e^{(a(t)-\frac{\eps}{2})\abs{x}^2}u(t)},
  \label{eq:error1}\\
  \norma{B_{\eps,\delta}}_{L^2(\Rn\times[-1+\delta,1-\delta])}\leq & \, N_{a,\mathbf{b},A,\eps}\sup_{t\in[-1,1]}\norma{e^{(a(t)-\frac{\eps}{2})\abs{x}^2}u(t)}.
  \label{eq:error2}
\end{align}
Moreover
\begin{equation}\label{eq:errorV}
  \sup_{t\in[-1,1]}\norma{(V f_\eps)\ast\theta_\delta(t)}\leq \norma{V}_{L^\infty(\Rn\times[-1,1])}\sup_{t\in[-1,1]}\norma{e^{(a(t)-\frac{\eps}{2})\abs{x}^2}u(t)}.
\end{equation}
The function $g_{\eps,\delta}$ verifies analogous relations, obtained setting $\mathbf{b}\equiv 0$ in the previous equations.

Since the $f_{\eps,\delta}$ and $g_{\eps,\delta}$ are now regular, \eqref{eq:astrattoutile} holds. Therefore,
\begin{equation}\label{eq:serve1}
  \begin{split}
      \partial_t&\left(\frac{1}{a_\eps} \, \partial_t H_{\eps,\delta} 
        -\frac{1}{a_\eps} \, \Re(\partial_t g_{\eps,\delta} -
      \mathcal{S_\eps} g_{\eps,\delta} 
      - \mathcal{A_\eps}g_{\eps,\delta},g_{\eps,\delta})\right) \\
      & + \frac{1}{a_\eps} \, \norma{\partial_t g_{\eps,\delta}
        -\mathcal{S_\eps}g_{\eps,\delta}-\mathcal{A_\eps}g_{\eps,\delta}}^2\\ 
      \geq& \, 2 \left(\frac{1}{a_\eps}  \mathcal{S_\eps}_t g_{\eps,\delta}
       + \frac{1}{a_\eps} \, [\mathcal{S_\eps},\mathcal{A_\eps}]g_{\eps,\delta} 
       - \frac{\dot a_\eps}{a_\eps^2}  \,
      \mathcal{S_\eps}g_{\eps,\delta},g_{\eps,\delta}\right).
    \end{split}
\end{equation}
Moreover, from \eqref{eq:richiamero} in Lemma \ref{lemma3}, with $\gamma=1/a_{\eps}$ and $\mathbf{b}\equiv 0$, we get
\begin{equation}\label{eq:serve2}
    \begin{split} 
      \int_{\Rn} & \left( \frac{1}{a_\eps} \mathcal{S_{\eps}}_t + \frac{1}{a_\eps} [\mathcal{S_\eps},\mathcal{A_\eps}]
      -\frac{\dot a_\eps}{a_\eps^2} \mathcal{S_\eps} \right) g_{\eps,\delta} \bar g_{\eps,\delta}\, dx  \\
                 =& \int_{\Rn} 8 \left\vert -i\nabla_A g_{\eps,\delta} 
                   + \left(\frac{\dot a_\eps }{4a_\eps}\right)xg_{\eps,\delta} \right\vert ^2\,dx 
                  +\int_{\Rn}F(a_\eps)\abs{x}^2\abs{g_{\eps,\delta}}^2\,dx\\ 
                  &-8 \int_{\Rn}\Im(g_{\eps,\delta} \, x^t B \overline{\nabla_A g_{\eps,\delta}})\,dx.
    \end{split}
  \end{equation}
Since $F(a_\eps)>0$,
 there exists a constant $C>0$ depending on $a$, such that we have

\begin{equation}\label{eq:dariusareallafine}
  \begin{split}
    & \int_{\Rn} 8 \left\vert -i\nabla_A g_{\eps,\delta} 
      + \left(\frac{\dot a_\eps }{4a_\eps}\right)xg_{\eps,\delta} \right\vert ^2\,dx 
    +\int_{\Rn}F(a_\eps)\abs{x}^2\abs{g_{\eps,\delta}}^2\,dx  \\
    &\geq \int_{\Rn} \left\vert \nabla_A \left(e^{\frac{i \dot a_\eps }{8 a_\eps}\abs{x}^2}g_{\eps,\delta}\right)\right\vert^2\,dx
    +C_a\int_{\Rn} \abs{\nabla_A g_{\eps,\delta}}^2 + \abs{x}^2\abs{g_{\eps,\delta}}^2\,dx. \\
  \end{split}
\end{equation}
Moreover there exists an arbitrarily small $\eta>0$ such that
\begin{equation}\label{eq:serve3}
    -8 \int_{\Rn}\Im(g_{\eps,\delta} \, x^t B \overline{\nabla_A g_{\eps,\delta}})\,dx \geq 
     -\frac{16}{\eta}\norma{x^t B}_{L^\infty}^2\int_{\Rn}\abs{g_{\eps,\delta}}^2\,dx 
    - \eta \int_{\Rn}\abs{\nabla_A g_{\eps,\delta}}^2\,dx.
\end{equation}
By \eqref{eq:serve1},  \eqref{eq:serve2}, \eqref{eq:dariusareallafine},  \eqref{eq:serve3}, we get
\begin{equation}\label{eq:dariusareallafine}
  \begin{split}
    &\int_{\Rn} \left\vert \nabla_A \left(e^{\frac{i \dot a_\eps }{8 a_\eps}\abs{x}^2}g_{\eps,\delta}\right)\right\vert^2\,dx+
    C\int_{\Rn} \abs{\nabla_A g_{\eps,\delta}}^2 + \abs{x}^2\abs{g_{\eps,\delta}}^2\,dx \\
    & \ \ \ \leq 
      \partial_t\left(\frac{1}{a_\eps} \, \partial_t H_{\eps,\delta} 
        -\frac{1}{a_\eps} \, \Re(\partial_t g_{\eps,\delta} -
      \mathcal{S_\eps} g_{\eps,\delta} 
      - \mathcal{A_\eps}g_{\eps,\delta},g_{\eps,\delta})\right) \\
      & \ \ \ \ \ \ + \frac{1}{a_\eps} \, \norma{\partial_t g_{\eps,\delta}
        -\mathcal{S_\eps}g_{\eps,\delta}-\mathcal{A_\eps}g_{\eps,\delta}}^2 
      + D \norma{x^t B}_{L^\infty}^2 H_{\eps,\delta},
  \end{split}
\end{equation}
for some constants $C,D>0$ depending on $a$.
Multiplying the last inequality by $(1-\delta_\eps)^2 - t^2$, and integrating by part in time,
we get
\begin{equation*}
  \norma{\sqrt{(1-\delta_\eps)^2-t^2} \nabla_A g_{\eps,\delta}}_{L^2(\Rn\times [-1+\delta_\eps,1-\delta_\eps])} \leq N_{a,B,\eps},
\end{equation*}
and analogously 
\begin{equation*}
  \norma{\sqrt{(1-\delta_\eps)^2 - t^2} \nabla_A f_{\eps,\delta}}_{L^2(\Rn\times [-1+\delta_\eps,1-\delta_\eps])} \leq N_{a,\mathbf{b},B,\eps},
\end{equation*}
thanks to \eqref{eq:errori}, \eqref{eq:error1}, and \eqref{eq:error2}.
Letting $\delta$ tend to zero, we find that
\begin{equation*}
  \norma{\sqrt{(1-\delta_\eps)^2-t^2} \nabla_A f_{\eps}}_{L^2(\Rn\times [-1+\delta_\eps,1-\delta_\eps])} \leq N_{a,\mathbf{b},B,\eps},
\end{equation*}
which makes possibile to integrate in time by parts the first term in $B_{\eps,\delta}$, obtaining
\begin{equation*}
  \begin{split}
    B_{\eps,\delta}(x,t)=&-\int_{\Rn} \nabla_y f_\eps (y,s) \cdot [4i(a_\eps(s)(y+\mathbf{b}(s))-a_\eps(t)(x+\mathbf{b}(t)))
    ]\theta^{x,t}_\delta(y,s)\,dyds \\
    &-\int_{\Rn} \nabla_y f_\eps (y,s) \cdot [2 (A(x)-A(y))]\theta^{x,t}_\delta(y,s)\,dyds  \\
    & +\int_{\Rn} f_\eps (y,s) [2in(a_\eps(t)-a_\eps(s))+(\nabla_y\cdot A(y)-\nabla_x\cdot A(x))]
    \theta^{x,t}_\delta(y,s)\,dyds.
  \end{split}
\end{equation*}
This, together with the fact that $A\in \mathcal C^{1,\varepsilon}_{\text{loc}}$, allows to get finally
\begin{equation}\label{eq:error3}
  \norma{B_{\eps,\delta}}_{L^2(\Rn,[-1+\delta_\eps,1-\delta_\eps])}\leq \delta N_{a,\mathbf{b},A,\eps},
\end{equation}
when $0<\delta\leq \delta_\eps$, which improves \eqref{eq:error2}.

Thanks to the above convergence results, we have that $f_\eps$ is in $C^{\infty}((-1,1),L^2(\Rn))$ and 
that $H_{\eps,\delta}$ converges uniformly on compact sets of $(-1,1)$ to $H_{\eps}(t)=\norma{f_\eps(t)}^2$.
Letting $\delta$ and $\eps$ tend to zero, we get finally
\begin{equation*}
  \norma{e^{a(t)\abs{x+\mathbf{b}(t)}^2}u(t)}^2 
\leq \sup_{t \in [-1,1]} 
\norma{e^{\mu\abs{x}^2}u(t)}e^{2T(t)+M(t)+4\norma{V}_{L^\infty}}
\end{equation*}
when $\abs{t}\leq 1$, with
\begin{equation*}
  \begin{sistema}
    \partial_t \left(\frac{1}{a}\partial_t M\right)= -\frac{1}{a}\norma{V}_{L^\infty}^2 \\
    M(-1)=M(1)=0.
  \end{sistema}
\end{equation*}
Notice that $M$ is even, and
\begin{equation*}
  M(t)=\norma{V}_{L^\infty}^2\int_t^1\int_0^s \frac{a(s)}{a(\tau)}\,d\tau ds, \quad \text{ in }[0,1],
\end{equation*}
and, since $a$ is monotone in $[0,1]$, we get the \eqref{eq:spostaregaussiana}.
Using again \eqref{eq:dariusareallafine}, analogously we have 
\begin{equation*}
\begin{split}
  &\norma{\sqrt{(1-\delta_\eps)^2-t^2}\nabla_A(e^{\frac{i\dot a _\eps}{8a_\eps}\abs{x}^2} 
    g_{\eps,\delta})}_{L^2(\Rn\times[-1+\delta_\eps,1-\delta_\eps])} \\
  &+C_a \norma{\sqrt{(1-\delta_\eps)^2-t^2}\nabla_A g_{\eps,\delta}}_{L^2(\Rn\times[-1+\delta_\eps,1-\delta_\eps])} \\
  &+C_a \norma{\sqrt{(1-\delta_\eps)^2-t^2}x g_{\eps,\delta}}_{L^2(\Rn\times[-1+\delta_\eps,1-\delta_\eps])} \\
  &\leq C e^{2\norma{V}_{L^\infty}+\frac{\norma{V}_{L^\infty}^2}{4}}\sup_{t \in [-1,1]} \norma{e^{\mu \abs{x}^2}u(t)}+ \delta N_{a,\eps, A, B},
\end{split}
\end{equation*}
for $C=C(\norma{V}_{\infty}, \norma{x^t B}_{\infty})$.
Letting $\delta$ and $\eps$ go to zero, we get \eqref{eq:disuguaglianzalemma4} and we conclude the proof.
\end{proof}

\section{Proof of Theorem \ref{thm:main}}\label{sec:main}
For convenience, we will denote by
\begin{align}
    \label{eq:assA3hardy} & M_B:=2\|x^tB\|_{L^\infty}^2 < +\infty,  \\
    \label{eq:assVhardy}  & M_V:=2\norma{V}_{L^\infty}+ \frac{\norma{V}^2}{4} < +\infty. 
\end{align}
The proof is divided into several steps.
  
\subsection{Cr\"onstrom gauge}\label{sec:cronstrom}
The first step consists in reducing to the Cr\"onostrom gauge 
\begin{equation*}
  x\cdot A(x)=0 \quad \text{ for all }x \in  \Rn,
\end{equation*} 
by means of the following result.

\begin{lemma}\label{lem:cronstrom1}
  Let $A=A(x)=(A^1(x),\dots,A^n(x)):\R^n\to\R^n$, for $n\geq2$
 and denote by $B=DA-DA^t\in \mathcal M_{n\times
n}(\R)$, $B_{jk}=A^k_j-A^j_k$, and $\Psi(x):=x^tB(x)\in\R^n$. 
Assume that the two vector quantities
  \begin{equation}\label{eq:cronstrom1}
    \int_0^1A(sx)\,ds\in\R^n,
    \qquad
    \int_0^1\Psi(sx)\,ds\in\R^n
  \end{equation}
  are finite, for almost every $x\in\R^n$; moreover, define the (scalar) function
  \begin{equation}\label{eq:varphi}
    \varphi(x):=x\cdot\int_0^1A(sx)\,ds\in\R.
  \end{equation}
  Then, the following two identities hold:
  \begin{align}\label{eq:cronstrom2}
    \widetilde A(x):=A(x)-\nabla\varphi(x)  & = -\int_0^1\Psi(sx)\,ds
    \\
    x^tD\widetilde A(x)  & = -\Psi(x) +\int_0^1\Psi(sx)\,ds.
    \label{eq:cronstrom3}
  \end{align}
\end{lemma}
\begin{remark}\label{rem:cronstrom}
Notice that
  \begin{equation}\label{eq:gaugefinal}
    x\cdot\widetilde A(x) \equiv0,
    \qquad
    x\cdot x^tD\widetilde A(x)\equiv0.
  \end{equation}
  From now on, we will hence assume, without loss of generality, that \eqref{eq:gaugefinal} are satisfied by $A$. Observe moreover that assumption \eqref{hypo:xi} in Theorem \ref{thm:main} is preserved by the above gauge transformation, and we have in addition that $A\cdot\xi\equiv0$. 
We also remark that
\begin{equation*} 
  \norma{\tilde A}_{L^\infty}^2 +   \norma{x^t B}_{L^\infty}^2 \leq  M_B. 
\end{equation*}
Finally notice that the first condition in \eqref{eq:cronstrom1} is guaranteed by the assumption $A\in \mathcal C^{1,\varepsilon}_{\text{loc}}$ in Theorem \ref{thm:main}.

We mention \cite{I} for the proof of the previous Lemma; see alternatively Lemma 2.2 in \cite{BFGRV}.
\end{remark}

\subsection{Appell Transformation}\label{sec:Appell}
Following the strategy in \cite{EKPV2, EKPV3, BFGRV}, the second step is to reduce assumption \eqref{eq:decay3hardy} to the case $\alpha=\beta$, by pseudoconformal transformation (Appell transformation).
\begin{lemma}[\cite{BFGRV}, Lemma 2.7]\label{lem:appell}
  Let $A=A(y,s)=(A^1(y,s),\dots,A^n(y,s)):\R^{n+1}\to\R^{n}$,
\mbox{$V=V(y,s)$}, $F=F(y,s):\R^n\to\C$, $u=u(y,s):\R^{n}\times[0,1]\to\C$ be a solution to
  \begin{equation}\label{eq:1appell}
  \partial_su=i\left(\Delta_Au+V(y,s)u+F(y,s)\right),
\end{equation}
and define, for any $\alpha,\beta>0$, the function
\begin{equation}\label{eq:appell}
  \widetilde u(x,t):=
  \left(\frac{\sqrt{\alpha\beta}}{\alpha(1-t)+\beta t}\right)^{\frac n2}
  u\left(\frac{x\sqrt{\alpha\beta}}{\alpha(1-t)+\beta t},\frac{t\beta}{\alpha(1-t)+\beta t}\right)
  e^{\frac{(\alpha-\beta)|x|^2}{4i (\alpha(1-t)+\beta t)}}.
\end{equation}
Then $\widetilde u$ is a solution to
\begin{equation}\label{eq:2appell}
  \partial_t\widetilde u=i\left(\Delta_{\widetilde A}\widetilde u
  +\frac{(\alpha-\beta)\widetilde A\cdot x}
  {(\alpha(1-t)+\beta t)}\widetilde u+
  \widetilde V(x,t)\widetilde u+\widetilde F(x,t)\right),
\end{equation}
where
\begin{align}
  \label{eq:Aappell}
  \widetilde A(x,t)
  &
  = \frac{\sqrt{\alpha\beta}}{\alpha(1-t)+\beta t}
  A\left(\frac{x\sqrt{\alpha\beta}}{\alpha(1-t)+\beta t},\frac{t\beta}{\alpha(1-t)+\beta t}\right)
  \\
  \label{eq:Vappell}
  \widetilde V(x,t)
  &
  = \frac{\alpha\beta}{(\alpha(1-t)+\beta t)^2}
  V\left(\frac{x\sqrt{\alpha\beta}}{\alpha(1-t)+\beta t},\frac{t\beta}{\alpha(1-t)+\beta t}\right)
  \\
  \label{eq:Fappell}
  \widetilde F(x,t)
  &
  = \left(\frac{\sqrt{\alpha\beta}}{\alpha(1-t)+\beta t}\right)^{\frac n2+2}
  F\left(\frac{x\sqrt{\alpha\beta}}{\alpha(1-t)+\beta t},\frac{t\beta}{\alpha(1-t)+\beta t}\right)
  e^{\frac{(\alpha-\beta)|x|^2}{4i
  (\alpha(1-t)+\beta t)}}.
\end{align}
\end{lemma}
\begin{remark}\label{rem:appell}
  The term containing $ \tilde A \cdot x$ vanishes (see Remark \ref{rem:cronstrom} above).
Moreover,
  assumptions \eqref{eq:assA3hardy} and \eqref{eq:assVhardy} 
  still hold for $\tilde B$ and $\tilde V$.  We finally remark that $\widetilde A$ is time-dependent. Nevertheless, notice that 
  \begin{equation}\label{hypo:grazieaCronstrom}
    x\cdot \widetilde A_t(x)= 0, \quad \xi \cdot \tilde A_t(x) = 0,
  \end{equation}  
for all $x \in \Rn, t \in [0,1]$.
\end{remark}
By direct computations, we have
  \begin{gather*}
    \left\Vert{e^{\frac{\abs{\cdot}^2}{\alpha\beta}}\widetilde u(0)}
    \right\Vert_{L^2}=\left\Vert{e^{\frac{\abs{\cdot}^2}{\beta^2}}u(0)}\right\Vert_{L^2},
     \quad \left\Vert{e^{\frac{\abs{\cdot}^2}{\alpha\beta}}\widetilde
        u(1)}
    \right\Vert_{L^2}=\left\Vert{e^{\frac{\abs{\cdot}^2}{\alpha^2}}u(1)}\right\Vert_{L^2},
    \\ 
    \sup_{t \in[0,1]}
    \left\Vert{e^{\frac{\abs{\cdot}^2}{\alpha\beta}}}\widetilde u(t)\right\Vert_{L^2}=\sup_{t \in[0,1]}
    \left\Vert{e^{\frac{\abs{\cdot}^2}{(\alpha t +
            \beta(1-t))^2}}u(t)}\right\Vert_{L^2}.
  \end{gather*} 

For convenience, we change the time interval in $[-1,1]$: let $v(x,t)= 2^{-\frac{n}{4}}\widetilde u
\big(\frac{x}{\sqrt{2}},\frac{1+t}{2}\big)$. 
The function $v$ is a solution to
\begin{equation*}
  \partial_t v = i(\Delta_{\mathcal A}v + \mathcal{V}v),\quad \text{ in }\Rn \times [-1,1],
\end{equation*} 
with 
\begin{equation*}
  \mathcal A (x,t)= \frac{1}{\sqrt{2}}A\left(\frac{x}{\sqrt{2}},\frac{1+t}{2}\right), 
  \quad \mathcal V (x,t)= \frac{1}{2}V\left(\frac{x}{\sqrt{2}},\frac{1+t}{2}\right).
\end{equation*}
The assumptions of Theorem \ref{thm:main} still hold (up to a change of the constants) and moreover
\begin{gather*}
    \left\Vert{e^{\frac{\abs{\cdot}^2}{2\alpha\beta}}v(0)}
    \right\Vert_{L^2}=\left\Vert{e^{\frac{\abs{\cdot}^2}{\alpha\beta}}\widetilde u(0)}
    \right\Vert_{L^2}=\left\Vert{e^{\frac{\abs{\cdot}^2}{\beta^2}}u(0)}\right\Vert_{L^2}, \\
    \left\Vert{e^{\frac{\abs{\cdot}^2}{2\alpha\beta}}v(1)}
    \right\Vert_{L^2}=\left\Vert{e^{\frac{\abs{\cdot}^2}{\alpha\beta}}\widetilde u(1)}
    \right\Vert_{L^2}=\left\Vert{e^{\frac{\abs{\cdot}^2}{\alpha^2}}u(1)}\right\Vert_{L^2},
    \\ 
    \sup_{t \in[-1,1]}
    \left\Vert{e^{\frac{\abs{\cdot}^2}{2\alpha\beta}}}v(t)\right\Vert_{L^2}=\sup_{t \in[0,1]}
    \left\Vert{e^{\frac{\abs{\cdot}^2}{\alpha\beta}}}\widetilde u(t)\right\Vert_{L^2}=\sup_{t \in[0,1]}
    \left\Vert{e^{\frac{\abs{\cdot}^2}{(\alpha t +
            \beta(1-t))^2}}u(t)}\right\Vert_{L^2}.
\end{gather*} 
We set
\begin{equation}\label{nualfabeta}
  \mu:=\frac{1}{2\alpha\beta}.
\end{equation}
The basic ingredient of our proof is the following logarithmic convexity estimate:
\begin{gather}\label{eq:convessitalogaritmica}
    \sup_{t \in[-1,1]}    \left\Vert e^{\mu\abs{\cdot}^2}v(t)\right\Vert_{L^2(\Rn)}=\sup_{t \in[0,1]}
    \left\Vert{e^{\frac{\abs{\cdot}^2}{(\alpha t +
            \beta(1-t))^2}}u(t)}\right\Vert_{L^2(\Rn)}  \\
   \leq  C \sup_{t\in[0,1]}\left\Vert e^{\frac{\abs{\cdot}^2}{\beta^2}}u(\cdot,0)\right\Vert^{\frac{\beta (1-t)}{\alpha t +\beta(1-t)}}_{L^2}
    \left\Vert e^{\frac{\abs{\cdot}^2}{\alpha^2}}u(\cdot,1)\right \Vert_{L^2}^{\frac{\alpha t }{\alpha t + \beta (1-t)}} \notag \\
    \leq C \left(\left\Vert e^{\frac{\abs{\cdot}^2}{\beta^2}}u(\cdot,0)\right\Vert_{L^2} +
    \left\Vert e^{\frac{\abs{\cdot}^2}{\alpha^2}}u(\cdot,1)\right \Vert_{L^2}\right)<+\infty, \notag 
\end{gather}
with 
\begin{equation*}
  C=C\left(\alpha,\beta,\norma{x^t B}_{L^\infty},\norma{V_1}_{L^{\infty}},
    \sup_{t\in[0,1]} \left\| e^{\frac{|\cdot|^2}{(\alpha t+\beta(1-t))^2}}V_2(\cdot,t)\right\|_{L^\infty}
    e^{\sup_{t\in[0,1]}\left\|\Im V_2(\cdot,t)\right\| _{L^\infty}}\right).
\end{equation*}
For the proof of \eqref{eq:convessitalogaritmica} see Theorem 1.5 in \cite{BFGRV}.
From now on, we denote $v$, $\mathcal A$ and $\mathcal V$ by $u$, $A$ and $V$.

We follow the same strategy as in \cite{EKPV3}, which is based on an iteration scheme. The argument here is a bit more delicate, due to the presence of additional terms involving the magnetic field.

\subsection{Conclusion of the Proof}\label{sec:conclusion}
We now apply an iteration scheme which is completely analogous to the one performed in \cite{EKPV3}.
The idea is to get the best possible choice for $a(t)$ such that an estimate like
\begin{equation}
  \label{eq:stimapesata}
  \norma{e^{a(t)\abs{x}^2}u(x,t)}_{L^\infty_t([-1,1])L^2_x(\Rn)}\leq
  C(\alpha,\beta,T,\norma{V}_{L^\infty},M_B).
\end{equation} 
holds. In order to do this, we will construct $a$ as the limit of an appropriate sequence $a_j(t)$, having in mind the improvement result of Lemma \ref{lem:lemma4}. At each step of the procedure, assumptions \eqref{eq:check1} and \eqref{eq:check2} have to be checked.  Also the curve $\mathbf b(t)=b(t)\xi$, with $\xi\in\mathbb S^{n-1}$ as in \eqref{hypo:xi} is naturally involved in the following argument.

\subsubsection{\underline{Iteration scheme}} 
Let us first construct the iteration scheme.
Assume that $k$ even and smooth functions $a_j\colon [-1,1]\to
(0,+\infty)$ and $C_{a_j}>0$, $j=1,\dots,k$ have been generated, such that
\begin{equation}\label{eq:iteractionrequest}
\begin{cases}
  \mu\equiv a_1< a_2 < \dots <a_k \quad \text{ in }(-1,1), 
  \\ 
  \dot a_j \leq 0 \, \text{ in }[0,1], \quad F(a_j)>0 \, \text{ in }[-1,1], 
  \quad a_j(\pm 1)=\mu,
  \\
  \sup_{t \in[-1,1]}\norma{e^{a_j(t)\abs{\cdot}^2}u(\cdot,t)}\leq e^{M_B \int_0^1 s a_j(s)\,ds} 
  e^{M_V}\sup_{t \in[-1,1]}\norma{e^{\mu\abs{\cdot}^2}u(\cdot,t)},
  \\
   \norma{\sqrt{1-t^2}\nabla_A(e^{a_j+\frac{i\dot a_j}{8
          a_j}\abs{x}^2}u)}_{L^2(\Rn\times [-1,1])}
    + C_{a_j} \norma{\sqrt{1-t^2}e^{a_j(t)\abs{x}^2}\nabla_A u}_{L^2(\Rn\times [-1,1])} 
    \\
	\ \ \ \leq C e^{M_V}\sup_{t \in[-1,1]}\norma{e^{\mu\abs{\cdot}^2}u(\cdot,t)}, 
	\end{cases}
\end{equation} 
where $C=C(\norma{V}_{\infty},\norma{x^t B}_{\infty})>0$, for all $j=1,\dots,k$.

The construction is identical to the one in \cite{EKPV3}; we repeat it here for the sake of completeness.
In order to simplify notations, set $c_k:=a_k^{-\frac{1}{2}}$.  Let
$b_k\colon[-1,1] \to \R$ be the solution to
\begin{equation}\label{eq:ODEb}
  \begin{sistema} \ddot b_k=-\frac{F(a_k)}{a_k}=-2c_k(16 c_k^{-3}-\ddot
    c_k)
    \\ 
    b_k(\pm 1)=0.
  \end{sistema}
\end{equation}

Observe that $b_k$ is even and
\begin{equation}\label{eq:defnb} b_k(t)=\int_t^1 \int_0^s
  \frac{F(a_k(\tau))}{a_k(\tau)}\, d\tau ds \quad \text{ in }[-1,1];
\end{equation} 
moreover $\dot b_k<0$ in $(0,1]$. Apply now \eqref{eq:spostaregaussiana} in
Lemma \ref{lem:lemma4} with $a=a_k$ and $\bold b= b_k \eta$, for
$\eta\in \R \xi=\{p\xi\mid p \in \R\}$: we get
\begin{equation}\label{eq:dalLemma} \norma{e^{a_k(t)\abs{\, \cdot
        \,+b_k(t)\eta}^2}u(\cdot,t)}_{L^2(\Rn)}\leq e^{T_k(t)+M_V}\sup_{t \in[-1,1]}\norma{e^{\mu\abs{\cdot}^2}u(\cdot,t)}_{L^2(\Rn)},
\end{equation}
 with
\begin{equation*}
  \begin{sistema}
    \partial_t \left(\frac{1}{a}\partial_t
      T_k\right)=-\left(\frac{\abs{\ddot
          b_k}^2\abs{\eta}^2}{F(a_k)}+M_B\right)=-\left(\frac{
        F(a_k)\abs{\eta}^2}{a_k^2}+M_B \right) \quad \text{ in }[-1,1]
    \\
    T_k(\pm 1)=0.
  \end{sistema}
\end{equation*} 
$T_k$ is even and, remembering that
$a_k(s)\leq a_k(\tau)$ if $\tau\leq s$,
\begin{equation*}
  \begin{split} T_k(t)=&\int_t^1 \int_0^s \left(
      \frac{a_k(s)}{a_k(\tau)}\frac{
        F(a_k(\tau))\abs{\eta}^2}{a_k(\tau)}+a_k(s) M_B\right)\,d\tau ds 
    \\ 
    \leq & \, \abs{\eta}^2 \int_t^1 \int_0^s
    \frac{F(a_k(\tau))}{a_k(\tau)}\,d\tau ds + M_B\int_t^1 s a_k(s)\,ds
    \\
    = & \, b_k(t)\abs{\eta}^2 + M_B\int_t^1 s a_k(s)\,ds,
  \end{split}
\end{equation*}
 for $t \in (-1,1)$. Therefore the right hand side of
\eqref{eq:dalLemma} can be estimated as follows:
\begin{equation*}
    \int_{\Rn}  e^{2a_k(t)\abs{x+b_k(t)\eta}^2}\abs{u(t)}^2\,dx \leq 
      e^{b_k(t)\abs{\eta}^2}e^{ M_B \int_t^1 s a_k(s)\,ds}e^{M_V}
    \sup_{t \in[-1,1]}\norma{e^{\mu \abs{\cdot}^2}u(\cdot,t)}.
\end{equation*}
Consequently we obtain
\begin{equation}\label{eq:disuguaglianzachiave}
  \begin{split} 
    \int_{\Rn} & e^{2
      a_k(t)\abs{x}^2-2{\eta}^2b_k(t)(1-a_k(t)b_k(t))+
      4a_k(t)b_k(t)x \cdot \eta}\abs{u(t)}^2\,dx \\ 
    & \leq
    e^{ M_B \int_t^1 s a_k(s)\,ds}e^{M_V}
    \sup_{t \in[-1,1]}\norma{e^{\mu \abs{\cdot}^2}u(\cdot,t)}.
  \end{split}
\end{equation}

Notice that, since $a_k$ is continuous in $[-1,1]$,
we can estimate
\begin{equation*} e^{ M_B \int_t^1 s a_k(s)\,ds}\leq C_k<+\infty.
\end{equation*}

By \eqref{eq:disuguaglianzachiave}, the check to be performed is concerned with the sign of $1-a_k(0)b_k(0)$.

If $1-a_k(0)b_k(0)\leq 0$ then by \eqref{eq:disuguaglianzachiave} $u\equiv 0$ and the scheme stops. 

If $1-a_k(0)b_k(0)> 0$, then $1-a_k(t)b_k(t)>0$ for all $t\in [-1,1]$, because of the monotonicity
of $a_k$ and $b_k$. In this case, we define the $(k+1)-$th functions $a_{k+1}$ and $c_{k+1}$ as follows: 
\begin{equation}\label{eq:defnak1} a_{k+1}=\frac{a_k}{1-a_k b_k},
  \quad c_{k+1}=a_{k+1}^{-\frac{1}{2}}.
\end{equation} 
We prove that the new defined $a_{k+1}$ verifies the
requests \eqref{eq:iteractionrequest}. Indeed it is easily seen that
$a_{k+1}$ is even, $a_{k+1}(\pm)=\mu$, $a_k<a_{k+1}$ in $(-1,1)$,
$\dot a_{k+1}\leq 0$ in $[0,1]$. The proof that $F(a_{k+1})>0$ in
$[-1,1]$ deserves some comment: recall that
\begin{equation*} F(a_{k+1})=2c_{k+1}^{-1}(16c_{k+1}^{-3}-\ddot
  c_{k+1}),
\end{equation*} 
moreover, from \eqref{eq:defnak1},
\begin{equation*}
  \begin{split} 
    &c_{k+1}=(c_k^2-b_k)^{\frac{1}{2}},
    \\
    &\ddot c_{k+1}=c_{k+1}^{-3}\left(16 -\frac{\dot b_k^2}{4}+c_k \dot c_k \dot
      b_k - \dot c_k^2 b_k -16 c_k^{-2}b_k\right).
  \end{split}
\end{equation*}
 From \eqref{eq:iteractionrequest} and
\eqref{eq:defnb}, we get $\dot c_k \dot b_k \leq 0$ and $16 b_k
c_k^{-2}+b_k^2>0$ in $[-1,1]$, hence \mbox{$16 c_{k+1}^{-3}-\ddot
  c_{k+1} >0$}.  

Multiplying \eqref{eq:disuguaglianzachiave}
by $\exp(-2\eps b_k(t)\abs{\eta}^2)$, $\eps>0$ and integrating the
corresponding inequality on the line $\R\xi$, with respect to $\eta$, we get
\begin{equation}\label{eq:defnak1eps} 
\sup_{t \in[-1,1]}
  \norma{e^{a_{k+1}^\eps(t)\abs{\cdot}^2}u(\cdot,t)}\leq C_k (1+
  \eps^{-1})^{\frac{n}{4}}e^{M_V}
  \sup_{t \in[-1,1]}\norma{e^{\mu \abs{\cdot}^2}u(\cdot,t)},
\end{equation} 
with
\begin{equation*} a_{k+1}^{\eps}=\frac{(1+\eps)a_k}{1+\eps -a_k b_k}.
\end{equation*}
Thanks to \eqref{eq:defnak1eps}, we have
\begin{equation*}
  \sup_{t \in[-1,1]}\norma{e^{(a_{k+1}(t)-\eps)\abs{\cdot}^2}u(\cdot,t)}<
  +\infty, \quad \text{ for all }\eps>0.
\end{equation*} 
Using the previous estimate, we can conclude that \eqref{eq:iteractionrequest} holds  up to $j=k+1$,
thanks to Lemma \ref{lem:lemma4}.

\subsubsection{\underline{Application of the iteration scheme}}\label{subsec:mu14} 
Let us describe the first step of the iteration.
Choose $a_1(t)\equiv \mu$, for all $t \in
[-1,1]$: obviously \eqref{eq:iteractionrequest} hold. 
Set $b_1$ to be the solution to \eqref{eq:ODEb}, that is
\begin{equation*} b_1(t)=16 \mu (1-t^2), \quad t\in[-1,1].
\end{equation*} 
We need the following preliminary result, already proved in \cite{BFGRV}, which will be useful in the sequel.
\begin{lemma}[\cite{BFGRV}, Theorem 1.1]\label{lem:parziale}
  In the hypoteses of Theorem \ref{thm:main}, if $\alpha\beta\leq2$
  then $u\equiv 0$.
\end{lemma}
\begin{proof}
The condition $\alpha\beta\leq2$, namely $\mu\geq\frac14$ by \eqref{nualfabeta},
is equivalent to $1-a_1(0)b_1(0)\leq0$. Then $u\equiv0$ by the above arguments based on \eqref{eq:disuguaglianzachiave}, and the proof is complete.
\end{proof}
By means of the previous Lemma, we only need to consider the range $\alpha\beta>2$, i.e. $\mu<\frac14$.

Apply the above described iteration procedure. 
If there exists $k\in\N$ such that $1-a_k(0)b_k(0)\leq0$, then $u\equiv0$ and the procedure stops.
 If for all $k \geq 1$ we have
$1-a_k(0)b_k(0)>0$, the above described iteration produces an increasing sequence
$(a_k)_{k\geq 1}$ of functions verifying \eqref{eq:iteractionrequest}.
Set
\begin{equation*} a(t):=\lim_{k}a_{k}(t),\quad t\in[-1,1].
\end{equation*}
We now need to distinguish two cases.

{\bf Case 1: $\mathbf{\lim_k a_k(0)<+\infty}$.} In this case, from \eqref{eq:iteractionrequest} we have
\begin{equation*}
    \sup_{t \in[-1,1]}\norma{e^{a(t)\abs{\cdot}^2}u(\cdot,t)}\leq e^{M_B
      \int_0^1 s a(s)\,ds} e^{M_V}\sup_{t \in[-1,1]}\norma{e^{\mu\abs{\cdot}^2}u(\cdot,t)} .
\end{equation*}
\begin{equation*}
\begin{split}
  &\norma{\sqrt{1-t^2}\nabla_A(e^{(a+\frac{i \dot a}{8 a})\abs{x}^2}u)}_{L^2(\Rn\times[-1,1])}
  + C_a \norma{\sqrt{1-t^2}e^{(a(t)-\eps)\abs{x}^2}\nabla_A u}_{L^2(\Rn\times [-1,1])} \\
  &\leq C \sup_{t\in [-1,1]}\norma{e^{\mu \abs{x}^2}u(t)},
\end{split}
\end{equation*}
for some $C=C(\norma{V}_{\infty},\norma{x^t B}_{\infty})>0$.

Moreover, $a$ can be determined as the solution to a suitable ordinary differential equation (see \cite{EKPV3} for details). One has 
\begin{equation*} a(t)=\frac{R}{4(1+R^2 t^2)},
\end{equation*} 
where $R>0$ is such that
\begin{equation*} \mu=\frac{R}{4(1+R^2)}.
\end{equation*} 
This forces $\mu\leq \frac{1}{8}$.
Estimate \eqref{eq:weightedestimate} hence immediately follows after inverting the changes in Section \ref{sec:Appell}.

{\bf Case 2: $\mathbf{\lim_k a_k(0)=+\infty}$.} In this case, if $\int_0^1 s a(s)\,ds < +\infty$,
then \eqref{eq:iteractionrequest} forces $u\equiv 0$.  If otherwise $\int_0^1 s a(s)\,ds = +\infty$, we need a more
detailed analysis. For all $k\geq 1$, let $s_k$ be the maximum point
of $s a_k(s)$ in $[0,1]$: from \eqref{eq:iteractionrequest} we have
\begin{equation*}
  \begin{split} 
    &\infty > e^{2\norma{V}_{L^\infty} +
      \frac{\norma{V}^2}{4}}\sup_{t \in [-1,1]}\norma{e^{\mu\abs{\cdot}^2}u(\cdot,t)}
    \geq \int_{\Rn} e^{2 a_k(0)\abs{x}^2 -M_B\int_0^1 s
      a_k(s)\,ds}\abs{u(0)}^2\,dx \\ 
    &\geq \int_{\Rn} e^{2 a_k(0)\abs{x}^2 -M_B s_k a_k(s_k)\,ds}\abs{u(0)}^2\,dx \geq 
    \int_{\Rn} e^{2 a_k(0)\left(\abs{x}^2-M_B \frac{s_k}{2}\right)}\abs{u(0)}^2\,dx.
  \end{split}
\end{equation*} 
If there exists a subsequence $(s_{k_{h}})_{h}$ such
that $s_{k_{h}}\to 0$, then the previous inequality implies that $u(0)\equiv0$ in $\R^n$, i.e. $u\equiv
0$. If no subsequences of $s_k$ accumulate in $0$, 
take $\bar s
>0$ a limit point of $(s_k)_k$: the previous inequality implies that
$u(0)\equiv0$ in the complementary of the ball centered in the origin of
radius $(M_B\bar s)/2$. As a consequence, by \eqref{eq:decay3hardy}, one can take $\beta>0$ arbitrarily small: then, by
Lemma \ref{lem:parziale}, we conclude that $u\equiv0$ in this case.

In conclusion, we summarize the above argument as follows: if $\mu>\frac18$, then necessarily we are either in the case 2 or in the case in which the scheme stops in a finite number of steps. In both cases, we proved that $u\equiv0$;
if $\mu\leq\frac18$, one can prove the logarithmic convexity estimates in \eqref{eq:weightedestimate}, by the arguments described in the case 1 above and the inversion of the changes of variables of Section \ref{sec:Appell}, for which we omit further details.

\end{document}